\documentclass{article}
\usepackage[utf8]{inputenc}
\usepackage[T1]{fontenc}
\usepackage[utf8]{inputenc}
\usepackage[margin=2.5cm]{geometry}

\usepackage{amsmath, amssymb, amsfonts, amsthm}
\usepackage{graphicx}
\usepackage{array}
\usepackage{multirow}
\usepackage[dvipsnames]{xcolor}
\usepackage{textpos}

\usepackage[colorlinks=true,citecolor = blue,linkcolor=blue]{hyperref}
\usepackage{cleveref}
\usepackage{thm-restate}

\newtheorem{theorem}{Theorem}[section]
\newtheorem{corollary}[theorem]{Corollary}
\newtheorem{definition}[theorem]{Definition}

\newtheorem{lemma}[theorem]{Lemma}

\newtheorem{observation}[theorem]{Observation}

\newtheorem{claim}[theorem]{Claim}

\newif\ifcomments
\commentstrue

\usepackage{authblk}
\newcommand{\notcross}[1]{$#1$-non-crossing}

\title{Planar induced paths via a decomposition into non-crossing ordered graphs}
\author[*]{Julien Duron}
\author[$\dagger$]{Hugo Jacob} 
\affil[*]{Institute of Informatics, University of Warsaw, Poland}
\affil[$\dagger$]{Department of Mathematics, TU Darmstadt, Darmstadt, Germany}

\date{}

\begin{document}

\maketitle

\makeatletter{\renewcommand*{\@makefnmark}{}
\footnotetext{*The work of Julien Duron on this project is supported by the
project BOBR that has received funding from the European Research Council (ERC)
under the European Union’s Horizon 2020 research and innovation programme
(grant agreement No 948057).}
\footnotetext{$\dagger$ The work of Hugo Jacob on this project is supported by
ANR GODASse ANR-24-CE48-4377 and by the Alexander von Humboldt Foundation in
the framework of the Alexander von Humboldt Professorship of Michał Pilipczuk,
endowed by the Federal Ministry of Research, Technology, and Space.}
\makeatother}

\begin{abstract}
In any graph, the maximum size of an induced path is bounded by the maximum
size of a path. However, in the general case, one cannot find a converse bound,
even up to an arbitrary function, as evidenced by the case of cliques. Galvin,
Rival and Sands proved in 1982 that, when restricted to weakly sparse graphs,
such a converse property actually holds.

In this paper, we consider the maximal function $f$ such that any planar graph
(and in general, any graph of bounded genus) containing a path on $n$ vertices
contains an induced path of size $f(n)$, and prove that $f(n) \in \Theta (\log
n/\log \log n)$ by providing a lower bound matching the upper bound obtained by
Esperet, Lemoine and Maffray, up to a constant factor. 
We obtain these tight bounds by analyzing graphs ordered along a Hamiltonian
path that admit an edge partition into a bounded number of sets without
crossing edges. In particular, we prove that the guaranteed size of an induced
path when the ordered graph can be partitioned into $2k$, resp. $2k+1$, sets of
non-crossing edges, is $\Theta_k((\log n/\log \log n)^{1/k})$, resp.
$\Theta_k((\log n)^{1/(k + 1)})$.
\end{abstract}

\paragraph{Acknowledgements.} We thank Daniel Gonçalves for his involvement in
the discussions which led to this paper. We also thank the organizers of the
first workshop of ANR GODASse where we could meet and establish the first
results. Finally, we wish to warmly thank the anonymous reviewers for their
feedback and suggestions, in particular on the approach to get tight bounds for
both odd and even number of non-crossing sets.

\begin{textblock}{20}(-1.75, 3.5)
\includegraphics[width=40px]{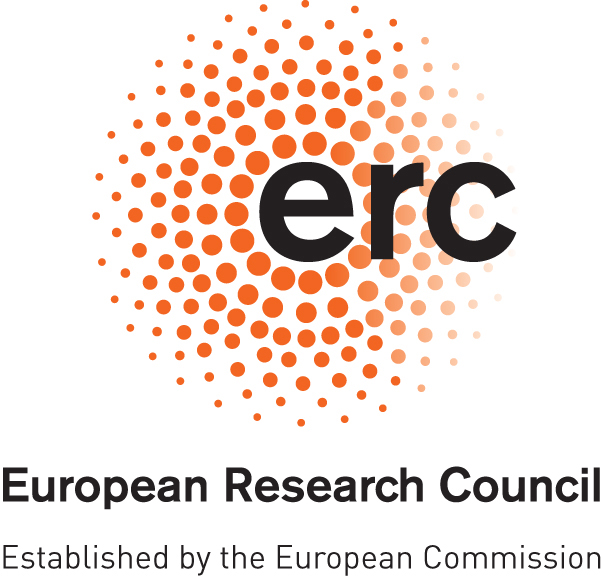}%
\end{textblock}
\begin{textblock}{20}(-1.75, 4.5)
\includegraphics[width=40px]{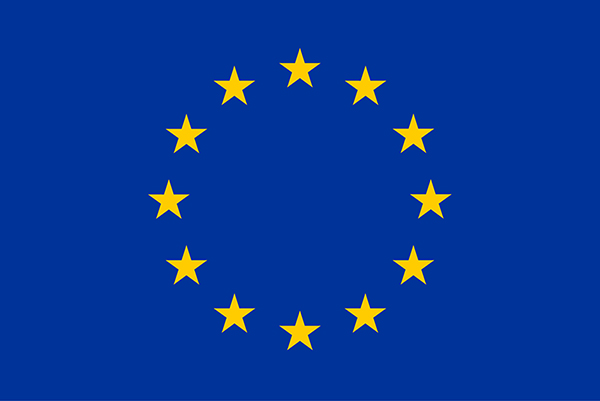}%
\end{textblock}

\section{Introduction}
This paper is dedicated to answering some questions related to the `guaranteed'
size of longest induced paths in graphs with sufficiently long paths. The first
result on this topic, obtained by Galvin, Rival and Sands \cite{GalvinRS82},
states that, for any $t \geqslant 2$, there is an unbounded function $f_t$ such
that any $K_{t,t}$-subgraph-free graph containing a path of size $n$ contains
an induced path of size $f_t(n)$. This was independently discovered later in
\cite{AtminasLR12}, and the bounds on the growth function $f_t$ have been
studied further in \cite{InducedPathSparse, DefrainR24, HunterMST24,
couetoux2025quasi}. To the best of our knowledge, the current best lower bound
$f_t = \Omega_t(\frac{\log \log n}{\log \log \log n})$ was obtained in
\cite{HunterMST24} while an almost matching upper bound $f_t = O(\log \log n
\cdot \log \log \log n)$ for $t\geqslant 3$ was proven recently in
\cite{couetoux2025quasi}.

Finding the right dependence function between the size of paths and induced
paths can be extended to any class of graphs $\mathcal{C}$: what is the largest
function $f_\mathcal{C}$ such that any graph in $\mathcal{C}$ containing a path
on $n$ vertices (as a subgraph) has to contain a path on $f_\mathcal{C}(n)$
vertices as an induced subgraph? We refer the reader to
\cite{InducedPathSparse} for a survey of the known bounds.
The special case of the class $\mathcal{P}$ of planar graphs has been
repeatedly studied in the literature \cite{ArochaV00, GiacomoLM16, EsperetLM17,
InducedPathSparse}, with a best known lower bound of  $f_{\mathcal{P}} =
\Omega(\sqrt{\log n})$ and a best upper bound of $f_{\mathcal{P}} = O(\log n /
\log\log n)$ both obtained in \cite{EsperetLM17}. We prove that
$f_{\mathcal{P}} = \Theta(\log n / \log\log n)$, improving the lower bound to
match the upper bound up to a (relatively small) constant factor. Our proofs
provide tight bounds in the general case of bounded genus graphs.

An important observation is that when $\mathcal{C}$ is a hereditary class, that
is a class closed under taking induced subgraphs, the function
$f_{\mathcal{C}}$ is equal to $f_{\mathcal{C}'}$ where $\mathcal{C}'$ denotes
the class of graphs in $\mathcal{C}$ containing a Hamiltonian path. Indeed, one
can work on the subgraph induced by the guaranteed path of size $n$.
In addition, as observed in \cite{InducedPathSparse}, ordered graphs naturally
arise from graphs containing a Hamiltonian path, by simply ordering the
vertices along this path. In this paper, we will follow this line of work, and
exhibit induced paths that are \emph{increasing} with respect to the order
given to our graphs.

Considering ordered graphs allows one to make precise arguments for graph
classes excluding some ordered subgraphs. A typical example of an ordered
subgraph one could exclude is the \emph{crossing}, that is four vertices $a
\prec b \prec c \prec d$ with the two edges $ac$ and $bd$, see
\Cref{fig:crossing}. Note that an ordered graph excluding a crossing is in
particular outerplanar.
In \cite{InducedPathSparse}, for graphs avoiding a crossing along a Hamiltonian
path, a tight lower bound of $\frac{\log n}{2}$ on the size of increasing
induced paths was obtained, recovering in particular the known lower bound on
outerplanar graphs.

\begin{figure}
    \centering
    \includegraphics{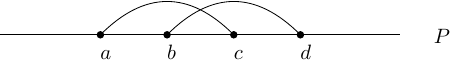}
    \caption{Illustration of a crossing: the edges $ac$ and $bd$ cross with respect to the order along $P$.}
    \label{fig:crossing}
\end{figure}

To prove a tight bound on planar graphs, we introduce \notcross{k} graphs,
which are ordered graphs for which there exists a partition of the set of edges
into $k$ parts, none of which contains a crossing.
This definition comes from the fact that finding long induced paths in planar
graphs with a Hamiltonian path can be reduced to finding long induced paths in
\notcross{2} graphs ordered along a Hamiltonian path (see
\Cref{lem:planar2simple}).
One can note that such a partition provides a book embedding (see e.g.
\cite{graphdrawingtextbook}), however since the vertex ordering is already
fixed by the Hamiltonian path, graphs admitting such a partition are much more
restricted than graphs of stack number $k$ (the analog of \notcross{k} when the
ordering can be chosen). For example, consider the (increasing) path on vertex
set $[2n]$, on which we add the edges between vertices $i$ and $i + n$ for $i
\in [n]$. This (unordered) graph admits a book embedding on 2 pages, but if it
is considered as an ordered graph with the natural order of $[2n]$, then it is
not \notcross{(n - 1)}.

Given a \notcross{k} (ordered) graph $G$, one can easily prove that for any set
$\{v_1, \dots, v_{k+1}\}$ of $k + 1$ vertices, there exists a large interval $I
\subseteq V(G)$ of vertices and at least one vertex $v_i$ such that $v_i$ is
not adjacent to any vertex of $I$.
Using the hypothesis that the size of induced paths is sufficiently small, we
strengthen this property to work for sets of at most $k$ vertices; see
\Cref{lem:gap} and \Cref{lem:gap_k}.
From that fact, we describe a procedure to grow a set of induced paths (or
trees depending on $k$) which ends only after a certain number of steps.
The intuition behind this procedure is the following. Consider a set of $k$
induced paths (not necessarily mutually induced), each one of them having a
marked endpoint denoted the \emph{root} of the path.
We furthermore assume the existence of some large interval $I$ of vertices that
is adjacent to all the roots, and non-adjacent to the rest of the paths.
From $I$, and using the fact that we have $k$ different roots, we find a
sub-interval $I'$ that is not adjacent to at least one of the roots, say $v$.
We then add to the path containing $v$ the neighbor $u \in I$ of $v$ that is
the closest to $I'$, and define $u$ as the new root of this path.

In \Cref{sec:planar}, we give a proof of a tight lower bound for the planar case which is based on a construction for graphs decomposing into two non-crossing graphs ordered along a Hamiltonian path. This serves both as a simpler self-contained proof for the class of planar graphs which is of particular interest, and as an introduction to the scheme later generalized to \notcross{k} ordered graphs. We also show how the case of graphs of bounded genus can be reduced to the planar case.

In \Cref{sec:k-outerplanar}, we give a proof of a lower bound for \notcross{k} ordered graphs. As an intermediate step, we use a polynomial bound between the size and the depth of trees of bounded cutwidth.

In \Cref{sec:upper-bounds}, we give upper bounds almost matching the lower bound of \Cref{sec:k-outerplanar}.

\section{Preliminaries}

In this paper, $\log$ denotes the binary logarithm and $[n]$ denotes $\{1,\dots,n\}$.

Given a graph $G$, we denote by $V(G)$ and $E(G)$ its vertex and edge sets.
Given a vertex $v \in V(G)$, $N_G(v)$ denotes the set of vertices adjacent to $v$ in $G$. We extend this notation to any subset $U \subseteq V(G)$ so that $N_G(U)$ is the set of vertices in $V(G) \setminus U$ adjacent to at least one vertex in $U$. $N_G[v]$ denotes $N_G(v) \cup \{v\}$. We drop subscripts when the graph is clear from context.
The \emph{size} of a path is its number of vertices.

An \emph{ordered graph} $(G,\prec)$ is a graph\footnote{We consider only simple graphs, i.e. undirected graphs with neither loops nor multi-edges.} $G$ with a given total order on its vertex set.
The \emph{cutwidth} of an ordered graph $(G,\prec)$ with ordering $v_1 \prec \dots \prec v_{|V(G)|}$ is the maximum over $i$ of the number of edges with one endpoint in $\{v_1,\dots,v_i\}$ and one endpoint in $\{v_{i+1},\dots ,v_{|V(G)|}\}$. The cutwidth of an unordered graph $G$ is the minimum over all orderings~$\prec$ of $V(G)$ of the cutwidth of $(G, \prec)$.

Given a graph $G$ with a Hamiltonian path $P := (v_1, \dots, v_{|V(G)|})$, the ordered graph $(G, \prec)$ is said to be \emph{ordered along the Hamiltonian path $P$}, or simply \emph{ordered along $P$}, when $\prec$ satisfies $v_i \prec v_j$ if and only if $i < j$.
An ordered graph $(G, \prec)$ with an increasing Hamiltonian path is in particular ordered along a Hamiltonian path.
A \emph{crossing} in an ordered graph $(G, \prec)$ is a pair of edges $(ab, cd)$ of $G$ such that $a \prec c \prec b \prec d$, see \Cref{fig:crossing}. When a set of edges $F$ of a given ordered graph $(G, \prec)$ does not contain a crossing, we say that $F$ is \emph{non-crossing}; and we say that $(G, \prec)$ is non-crossing whenever $E(G)$ is non-crossing.
It is well-known that a graph $G$ is outerplanar if and only if it admits an ordering $\prec$ such that $(G,\prec)$ is non-crossing.
Given an ordered graph $(G, \prec)$, a set of edges $F \subseteq E(G)$ is said to be \emph{\notcross{k}} when $F$ can be partitioned into $k$ sets $F_1 \cup \dots \cup F_k$, each being non-crossing. When $E(G)$ is \notcross{k}, we say that $(G, \prec)$ is \notcross{k}.

Given an ordered graph $(G, \prec)$ with vertices $v_1 \prec v_2 \prec \dots \prec v_{|V(G)|}$, an \emph{interval} of $V(G)$ is a subset of the form $\{v_i, \dots, v_j\}$ with $i \leqslant j$.
Such an interval is denoted by $[v_i, v_j]$. Its \emph{size} is its number of vertices. We extend the interval notation to the variants $[a,b[$, $]a,b]$ and $]a,b[$ where some endpoints are not in the interval. We denote by $(G,\prec)[v_i,v_j]$ the ordered subgraph induced by the interval $[v_i,v_j]$.
Given a partition $\mathcal{P}$ of $V(G)$ into intervals $\mathcal{P} = \{I_1, \dots, I_{|\mathcal{P}|}\}$, the \emph{quotient ordered graph} $(G, \prec)/\mathcal{P}$ is the ordered graph with vertex set $\mathcal{P}$, with an edge between intervals $I_a$ and $I_b$ if and only if there exists an edge between a vertex of $I_a$ and a vertex of $I_b$, and with total order $\prec_{\mathcal{P}}$ (denoted also $\prec$ when $\mathcal{P}$ is clear from the context) where $I_a \prec_{\mathcal{P}} I_b$ if and only if there exist $x \in I_a$ and $y \in I_b$ with $x \prec y$.
Note that, if $(G, \prec)$ contains an increasing Hamiltonian path, then so does $(G, \prec)/\mathcal{P}$.

Consider an ordered graph $(G, \prec)$ with $v_1 \prec v_2 \prec \dots \prec v_{|V(G)|}$.
Given a vertex $u$ and an interval $I$ of $V(G)$, the \emph{neighborhood gap of $u$ in $I$} is the largest integer $k$ such that there exist integers $i, j$ with $i + k \leqslant j + 1$ such that $v_i$ and $v_j$ are in  $I$, and $u$ is not adjacent to $[v_i, v_j]$. In other words, $k$ is the size of the largest interval of $I$ which is not hit by $N(u)$. We say that \emph{$u$ is adjacent to $I$} if there exists $v \in I$ such that $uv$ is an edge.

Given an ordered graph $(G, \prec)$ and a partition $\mathcal{P}$ of $V(G)$ into intervals, one can exhibit an induced path of $G$ from an increasing induced path of $(G, \prec)/\mathcal{P}$. However, our proofs will obtain quotient graphs with almost no edges, which will allow us to exhibit an increasing induced path instead of just an induced path. This fact is formalized by the following simple lemma.  

\begin{lemma}\label{lem:lifting}
Let $(G, \prec)$ be ordered along a Hamiltonian path. For any partition $\{I_1, \dots, I_k\}$ of $V(G)$ into intervals such that $I_1 \prec I_2 \prec \dots \prec I_k$, and such that there is an edge between $I_i$ and $I_j$ only when $|i - j| \leqslant 1$, $(G, \prec)$ contains an increasing induced path of size $k$.
\end{lemma}
\begin{proof}
Indeed, consider a shortest increasing path $P = a_1, \dots, a_\ell$ of $(G, \prec)$ from a vertex of $I_1$ to a vertex of $I_k$. Such a path exists since $G$ contains a Hamiltonian increasing path. Moreover, assume for the sake of contradiction that an interval $I_j$ does not contain a vertex of $P$. Then there is an index $i$ with $a_i \in I_1 \cup \dots \cup I_{j-1}$ and $a_{i+1}$ in $I_{j+1} \cup \dots \cup I_k$; a~contradiction.
\end{proof}

We will later refer to an ordered graph of cutwidth at most $1$ as an \emph{increasing sequence of increasing paths} due to the following observation.

\begin{observation}\label{obs:cw1}
The connected ordered graphs of cutwidth equal to 1 are the increasing paths.
The ordered graphs of cutwidth at most $1$ (with no isolated vertex) are their ordered subgraphs. 
\end{observation}

\section{Planar graphs and \notcross{2} graphs}\label{sec:planar}

This section is dedicated to the proof of the following theorem.

\begin{theorem}\label{thm:planar-lower-bound}
Any $n$-vertex planar graph $G$ with a Hamiltonian path $P$ admits an induced path on $\frac{\log n}{2\log \log n} - O(1)$ vertices which is increasing with respect to $P$.
\end{theorem}

\Cref{thm:planar-lower-bound} follows from \Cref{lem:planar2simple} and \Cref{thm:2outerplanar-lower-bound}, the former reducing the problem to finding increasing induced paths in \notcross{2} graphs ordered along a Hamiltonian path, and the latter proving the bound for these ordered graphs. Note that applying \Cref{lem:planar2simple} either directly gives the desired induced path, or reduces the graph size by a logarithmic factor, which does not affect the final conclusion.

\begin{lemma}\label{lem:planar2simple}
Let $(G,\prec)$ be an $n$-vertex planar graph ordered along the Hamiltonian path $v_1 \prec \dots \prec v_n$.
For any positive integer $k$, either $(G,\prec)$ contains an induced increasing path on $k+1$ vertices or there exists $1 \leqslant i < j \leqslant n$ such that $j - i > \frac{n-1}{k}$ and $(G, \prec)[v_i, v_j]$ is \notcross{2}.
\end{lemma}
\begin{proof}
Let $v_iv_j$ be the edge of $E(G)$ that maximizes $j - i$.
If $j - i \leqslant \frac{n-1}{k}$, then let $Q=(v_{i_1},\dots,v_{i_\ell})$ be a shortest increasing path from $v_1$ to $v_n$.
We have that $n = i_\ell \leqslant 1 + (\ell-1) \cdot \frac{n-1}{k}$, which leads to $k \leqslant \ell - 1$. Hence, $Q$ is of size at least $k+1$.

Otherwise, consider $H := G[v_i, \dots, v_j]$. Since $v_iv_j$ is an edge, the cycle $C =(v_i,v_{i+1} ,\dots,v_j)$ is a Hamiltonian cycle of $H$.
Hence, in any plane drawing of $H$, the cycle $C$ separates the plane into two connected regions, say $R_1$ and $R_2$.
Let $(A, B)$ be the partition of $E(H)\setminus E(C)$ where $A$ contains the edges mapped to $R_1$, and $B$ contains the edges mapped to $R_2$. One can think about edges of $A$ and $B$ being \underline{a}bove and \underline{b}elow the path respectively.
By definition, since $H_A=(V(H), A \cup E(C))$ and $H_B=(V(H), B \cup E(C))$ are both outerplanar graphs with exterior face $C$, neither $H_A$ nor $H_B$ contains two crossing edges for the order $v_i \prec \dots \prec v_j$.
\end{proof}

\begin{lemma}\label{lem:gap}
Let $(G, \prec)$ be a \notcross{2} graph ordered along a Hamiltonian path $P$, and let $g$ be a positive integer. Then, if there exist vertices $u$ and $v$ and an interval $I$ such that $u \prec I \prec v$ and both $u$ and $v$ have a neighborhood gap less than $g$ in $I$, then $(G,\prec)$ contains an increasing induced path of size $\left\lfloor \frac{|I|}{g} \right\rfloor - 2$.
\end{lemma}
\begin{proof}
Let $(A, B)$ be a partition of $E(G)$ into two non-crossing sets of edges.
Let $I$ be an interval of $P$, and $u$ and $v$ be two vertices with neighborhood gap less than $g$ in $I$, such that $u \prec I \prec v$.
We partition $I$ into $p = \lfloor |I|/g \rfloor$ intervals $I_1, \dots, I_p$ each of size at least $g$, such that $I_1 \prec I_2 \prec \dots \prec I_p$. Note that we can assume $p \geqslant 4$ since otherwise the lemma is trivial. Observe that since the neighborhood gaps of $u$ and $v$ in $I$ are less than $g$, both $u$ and $v$ are adjacent to $I_i$ for every $i$.

First, we claim that the edges between $u$ and $I_2 \cup \dots \cup I_{p}$ are all in $A$ or all in $B$.
Indeed, if this is not the case, there exist $ux \in A$ and $uy \in B$ with both $x, y \in I_2 \cup \dots \cup I_{p}$.
Moreover, $v$ is adjacent to a vertex $z \in I_1$. Hence, if $zv \in A$, we have $ux$ and $zv$ that are crossing, and otherwise we have $zv \in B$ thus $zv$ and $uy$ are crossing (see \Cref{fig:forced-crossing}).
By a symmetric argument, we obtain that all edges from $v$ to $I_1 \cup \dots \cup I_{p-1}$ are all in $A$ or all in $B$.
Finally, since $p \geqslant 4$, we cannot have edges from $v$ to $I_2 \cup \dots \cup I_{p-1}$ and from $u$ to $I_2 \cup \dots \cup I_{p-1}$ both be in $A$ (resp. in $B$).
In summary, we can assume that edges from $u$ to $I_2 \cup \dots \cup I_{p-1}$ are in $A$ and edges from $v$ to $I_2 \cup \dots \cup I_{p-1}$ are in $B$ and that there is at least one edge between $u$ (resp. $v$) and each of $I_2,\dots,I_{p-1}$.

Let $(H, \prec)$ be the quotient graph $(G[I_2 \cup \dots \cup I_{p-1}], \prec)/ \{I_2, \dots, I_{p-1}\}$. We claim that any edge in $H$ is of the form $I_kI_{k+1}$.
Indeed, assume towards a contradiction that $I_xI_y$ is an edge of $H$ with $x + 1 < y$. Then there exist $a \in I_x$ and $b \in I_y$ such that $ab$ is an edge of $G$. In particular, we have $ab \in A$ or $ab \in B$.
Since $|I_{x+1}| \geqslant g$, there exists $w_1 \in I_{x+1}$ (resp. $w_2$) that is adjacent to $u$ (resp. to $v$).
Hence, depending on whether $ab \in A$ or $ab \in B$, either $ab$ and $uw_1$ are in $A$ and crossing, or $ab$ and $w_2v$ are in $B$ and crossing; a~contradiction.

This implies that $H$ is a simple path of size $p - 2$, which implies the existence of an induced path of size $p-2$ in $G$ by \Cref{lem:lifting}.
\end{proof}

\begin{figure}
    \centering
    \includegraphics{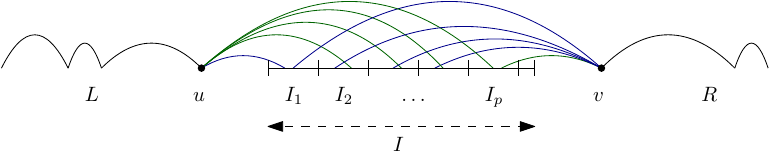}
    \caption{Illustration of the setup for \Cref{lem:gap} and \Cref{thm:2outerplanar-lower-bound}. The edges in $A$ are in blue, and the edges in $B$ are in green. Except for edges incident to $I_1$ and $I_p$, the edges incident to $u$ cannot be of the same color as the edges incident to $v$, otherwise they would cross.}
    \label{fig:forced-crossing}
\end{figure}

\begin{definition}
For any two vertices $u$ and $v$, for any interval $I$, we say that $(u, I, v)$ is a \emph{$k$-good triple} whenever 
\begin{itemize}
    \item $I \neq \varnothing$;
    \item $u \prec I \prec v$;
    \item $N(u) \cap I = \varnothing$ or $N(v) \cap I = \varnothing$;
    \item $G$ contains two increasing induced paths $L$ and $R$ such that the largest vertex of $L$ is $u$, the smallest vertex of $R$ is $v$, and $|L| + |R| = k$; and
    \item $N(L - u) \cap ]u,b] = N(R - v) \cap [a,v[ = \varnothing$, where $[a,b]=I$.
\end{itemize}
\end{definition}

\begin{claim}\label{cl:aux}
If $G$ contains a $k$-good triple $(u, I, v)$ and $|I| \geqslant \frac{\log n}{\log \log n}$, then either~$G$ contains an increasing induced path on $\left\lfloor \frac{\log n}{ \log \log n} \right\rfloor - 2$ vertices, or $G$ contains a $(k+1)$-good triple $(u', I', v')$ with $|I'| \geqslant \left\lfloor\frac{\log \log n}{\log n}|I| \right\rfloor$.
\end{claim}
\begin{proof}
Let $(u, I, v)$ be a $k$-good triple, and $g = \left\lfloor \frac{\log \log n}{\log n}|I| \right\rfloor > 0$.
Assume up to symmetry that $N(u) \cap I = \varnothing$.
We denote by $u'$ the largest neighbor of $u$ smaller than $I$ (note that it exists because of the neighbor of $u$ on the Hamiltonian path which cannot be in $I$ due to $N(u) \cap I = \varnothing$).
By \Cref{lem:gap} applied to $u',v$ and $I$,  either $G$ contains an increasing induced path of size $\left\lfloor \frac{|I|}{g} \right\rfloor - 2 \geqslant \left\lfloor \frac{\log n}{\log \log n} \right\rfloor - 2$ and the claim holds, or the maximum between the neighborhood gap of $u'$ in $I$ and the neighborhood gap of $v$ in $I$ is at least $g$.

If the neighborhood gap of $u'$ in $I$ is at least $g$, there exists $I'=[a',b'] \subseteq I=[a,b]$ such that $I'$ has size at least $g$ and $N(u') \cap I' = \varnothing$.
We claim that $(u', I', v)$ is a $(k+1)$-good triple.
By construction of $u'$ and $I'$, we have $u' \prec I' \prec v$ and $N(u') \cap I' = \varnothing$.
We consider $L' = L \cdot u'$ and $R' = R$.
We have $N(L) \cap ]u',b'] \subseteq (N(L-u) \cap ]u,b]) \cup (N(u) \cap ]u',a[) \cup (N(u) \cap I) = \varnothing$, and $N(R - v) \cap [a',v[ \subseteq N(R - v) \cap [a,v[ = \varnothing$.
Since $u' \in ]u,b]$ and $N(L-u) \cap ]u,b] = \varnothing$, $u'$ is adjacent only to $u$ in $L'$. Using that $L$ is an increasing induced path, we deduce that $L'$ is an increasing induced path. Finally,  $|R'| + |L'| = k+1$.

Otherwise, the neighborhood gap of $v$ in $I$ is at least $g$.
Consider again $I' \subseteq I$ such that $I'$ has size at least $g$ and $N(v) \cap I' = \varnothing$. Then using the exact same arguments, we obtain that $(u', I', v)$ is a $(k+1)$-good triple.
\end{proof}

\begin{theorem}\label{thm:2outerplanar-lower-bound}
Any $n$-vertex ordered graph $(G,\prec)$ ordered along a Hamiltonian path $P$ that is \notcross{2} admits an increasing induced path on $\frac{\log n}{2\log \log n} - O(1)$ vertices.
\end{theorem}

\begin{proof}
We may assume that $n \geqslant 4$.
Let $v_1 \prec \dots \prec v_n$ be the vertices of $G$ ordered along $P$.
Let $g = \left\lfloor \frac{\log \log n}{\log n} \cdot (n-2) \right\rfloor>0$. By \Cref{lem:gap} (applied with gap value $g$, $u=v_1,v=v_n$ and the interval $[v_2,v_{n-1}]$), either $G$ contains an increasing induced path on $\left\lfloor \frac{\log n}{\log \log n} \right\rfloor -2$ vertices, or one of $v_1$ and $v_n$ has neighborhood gap at least $g$ in $[v_2,\dots,v_{n-1}]$.
In the first case, we are already done so assume we are in the second case. Then, there exists an interval $I$ of size at least $g$ that is not adjacent to $v_1$ or not adjacent to $v_n$.
Hence, $(v_1, I, v_n)$ is a $2$-good triple with $|I| \geqslant g$.
By iteratively applying \Cref{cl:aux} starting from $(v_1, I, v_n)$, we either obtain an increasing induced path on $\left\lfloor \frac{\log n}{\log \log n} \right\rfloor -2$ vertices, or a sequence of $(i+2)$-good triples $(a_i, I_i, b_i)$ for $i \in [0,k]$ such that: $I_0=I$; $|I_k|<\frac{\log n}{\log \log n}$; and $|I_{i+1}| \geqslant \left\lfloor\frac{\log \log n}{\log n}|I_i| \right\rfloor \geqslant \frac{\log \log n}{\log n}|I_i| - 1$. The latter can be rewritten $|I_i| \leqslant \frac{\log n}{\log \log n}(|I_{i+1}| + 1)$. One easily deduces by induction on $i$ that:
\[|I_i| \leqslant \sum_{i \leqslant j \leqslant k} \left( \frac{\log n}{\log \log n}\right)^{k+1-j}\]
With $i=0$ and after reindexing, this gives \[
\frac{\log\log n}{\log n}(n-2)-1 \leqslant |I| = |I_0| \leqslant \sum_{1 \leqslant j \leqslant k+1} \left( \frac{\log n}{\log \log n} \right)^j \leqslant (k+1) \left( \frac{\log n}{\log \log n} \right)^{k+1}
.\]

The inequalities above imply $k \geqslant \frac{\log n}{\log \log n} - O(1)$. Indeed, if $k+1 \geqslant \frac{\log n}{\log \log n}$, we are done. Otherwise, we have $\frac{\log\log n}{\log n}(n-2)-1 \leqslant \left(\frac{\log n}{\log \log n}\right)^{k+2}$ which can also be rewritten into the desired bound. In particular, we obtain a $(k+2)$-good triple $(a, J, b)$. By definition of a $(k+2)$-good triple, there exist increasing induced paths $L$ and $R$ of total size $k+2$.
Considering the longest among $L$ and $R$ gives an increasing induced path of size $\frac{\log n}{2 \log \log n} - O(1)$.
\end{proof}

\begin{theorem}\label{thm:genus-lower-bound}
For any fixed $g \geqslant 0$, any $n$-vertex graph $G$ embedded in a surface of genus $g$ with a Hamiltonian path $P$ admits an induced path on $\frac{\log n}{2\log \log n} - O(1)$ vertices which is increasing with respect to $P$.
\end{theorem}

\begin{proof}
Consider an ordered graph $(G,\prec)$ ordered along a Hamiltonian path. We partition it into $g+1$ intervals $I_1,\dots,I_{g+1}$ of size at least $\left\lfloor \frac{n}{g+1} \right\rfloor$. Note that each interval induces an ordered graph with an increasing Hamiltonian path. There must exist $i \in [g+1]$ such that $G[I_i]$ is planar (i.e. of genus $0$). Indeed, otherwise all $g+1$ disjoint subgraphs $G[I_1],\dots,G[I_{g+1}]$ would be non-planar, but then the genus of the surface would be at least $g+1$ because the genus of a surface is at least the number of non-planar connected components (see \cite{BHKY62,SB77} for stronger statements in the orientable and non-orientable cases).

By applying \Cref{thm:planar-lower-bound} to $G[I_i]$, we obtain an increasing induced path whose number of vertices is \[\frac{\log (n/(g+1))}{2\log \log (n/(g+1))} - O(1) = \frac{\log n}{2\log \log n} - O(1).\qedhere\]
\end{proof}

We then easily derive the non-Hamiltonian case for bounded genus graphs (including planar graphs).

\begin{corollary}\label{coro:nonhamiltonian2crossing}
    If $G$ is a graph embeddable in a surface of genus $g$ with a path $P$ of size $n$,
    then it contains an induced path $P'$ on $\frac{\log n}{2\log \log n} - O(1)$ vertices. Moreover, we have $V(P') \subseteq V(P)$ and the vertices along $P'$ are in the same relative order as in $P$.
\end{corollary}

\begin{proof}
    Let $H=G[V(P)]$. $H$ is a graph on $n$ vertices with a Hamiltonian path and its genus is at most that of $G$.
    Simply apply \Cref{thm:genus-lower-bound} to $H$ and conclude.
\end{proof}

\section{The case of \notcross{k} ordered graphs}\label{sec:k-outerplanar}

We now delve into the analysis of \notcross{k} ordered graphs. The reason we separated the two proofs is that when one tries to grow a set of $k > 2$ paths, similarly to the good triples of \Cref{cl:aux}, it may happen that the paths ``merge''.
The merging forces us to grow sets of trees instead of sets of paths, shrinking the size of obtained induced paths to less than $(\log n)^{2/k}$.
However, these new complications are actually necessary to study the behavior of induced paths in \notcross{k} graphs: as we will see in \Cref{sec:upper-bounds}, one cannot hope for a better bound than $(\log n)^{2/k}$. 

\begin{lemma}\label{lem:gap_and_interval}
For any two integers $g,  k \geqslant 1$, there is an integer $N := f_{\ref{lem:gap_and_interval}}(g, k)$ such that the following holds.
For any $(G,\prec)$ ordered along a Hamiltonian path and \notcross{k}, with $E(G)= E_1 \cup \dots \cup E_k$ a partition into non-crossing sets, for any interval $I$ of $G$ with $|I| \geqslant N$, and any $k$ distinct vertices $x_1, \dots, x_k \in V(G) \setminus I$,
if each $x_i$ has neighborhood gap less than $g$ in $I$,
then there is a sub-interval $I' \subseteq I$ with $|I'| \geqslant \frac{|I|}{k!} - 5kg$ such that, up to reindexing of $E_1, \dots, E_k$, all edges between $x_i$ and $I'$ are in $E_i$.
\end{lemma}
\begin{proof}
We prove the lemma by induction on $k$, with $f_{\ref{lem:gap_and_interval}}(g, k) := k! \cdot 2^k \cdot 6g$.
If $k = 1$, then the statement trivially holds since the edges of $E(G)$ are all in $E_1$.

We now consider $k>1$ and assume the result holds for any $k' < k$.

We partition $I$ into $p := \lfloor |I|/g \rfloor$ intervals $I_1, \dots, I_p$ each of size at least $g$, with $I_1 \prec I_2 \prec \dots \prec I_p$.
By construction, for any $i \in [k]$ and $j \in [p]$, $x_i$ is adjacent to $I_j$. 
In particular, there exist integers $e(1), \dots, e(p)$ in $[k]$ such that, for each $t \in [p],$ there is an edge between $x_k$ and $I_t$ in $E_{e(t)}$.
Observe that, for any integers $a < b < c$ such that $e(a) = e(c)$, none of the edges between a vertex $x_i$ with $i \neq k$ and $I_b$ are in $E_{e(a)}$, since such an edge would cross an edge of $E_{e(a)}$ incident to $x_k$.
Furthermore, since the sequence $e(1), \dots, e(p)$ has values in $[k]$, the pigeonhole principle implies that there exist $a \leqslant c$ with $c - a + 1 \geqslant p/k$ such that $e(a) = e(c)$. Indeed, the maximum distance between $a$ and $c$ such that $e(a)=e(c)$ is minimized when identical values of $e(t)$ are consecutive and the maximum number of identical values is minimized.
Up to exchanging $E_{e(a)}$ and $E_k$, one can assume that $e(a) = k$.

Let $I' := I_{a+1} \cup \dots \cup I_{c-1}$, and note that our assumption on $|I|$ gives
\[
|I'| \geqslant
g \cdot (c - a - 1) \geqslant
g \cdot (\frac{p}{k} - 2) \geqslant
g \cdot (\frac{|I|}{kg} - 3) \geqslant
\frac{|I|}{k} - 3g \geqslant
\frac{|I|}{2k} \geqslant f_{\ref{lem:gap_and_interval}}(g, k-1).
\]
Since vertices $x_1, \dots, x_{k-1}$ have neighborhood gap less than $g$ in $I'$, and no edge in $E_k$ is between them and $I'$, one can apply the induction: there exists $I'' \subseteq I'$ such that for any $i < k$, all edges between $x_i$ and $I''$ are in $E_i$.
Let $I'' = I''_l \cup I''_m \cup I''_r$ with $I''_l \prec I''_m \prec I''_r$ and $ |I''_l| = |I''_r| = g$. Since for each $i < k$, there is an edge of $E_i$ between $x_i$ and both $I''_l$ and $I''_r$, all the edges between $x_k$ and $I''_m$ are in $E_k$.
We have $|I''_m| \geqslant |I''| - 2g \geqslant \frac{|I'|}{(k-1)!} - 5 (k-1) g - 2g \geqslant \frac{|I|}{k!} - \frac{g + 2kg}{k!} - 5(k-1)g - 2g \geqslant \frac{|I|}{k!} - 5kg$.
Hence, $I''_m$ satisfies the statement of the lemma.
\end{proof}

\Cref{lem:gap_and_interval} has two applications. On the one hand, \Cref{cor:gap_k+} shows that in \notcross{k} ordered graphs, given any $k+1$ distinct vertices and an interval $I$, one of them has a large neighborhood gap in $I$. On the other hand, \Cref{lem:gap_k} allows one to find a win-win situation when given only $k$ vertices: either one finds a long induced path, or one of the vertices has a large neighborhood gap in $I$.

\begin{corollary}\label{cor:gap_k+}
For any integer $k$, there is an integer $N := f_{\ref{cor:gap_k+}}(k)$ such that the following holds.
For any $(G, \prec)$ ordered along a Hamiltonian path and \notcross{k}, for any interval $I$ of $G$ with $|I| \geqslant N$ and $k+1$ distinct vertices $x_1, \dots,x_{k+1} \in V(G) \setminus I$, one of the $x_i$ has neighborhood gap at least $|I|/c_{\ref{cor:gap_k+}}(k)$ in $I$ with $c_{\ref{cor:gap_k+}}(k) := 12\cdot k!\cdot 2^k$.
\end{corollary}
\begin{proof}
We will prove the corollary with $f_{\ref{cor:gap_k+}}(k) := 2 \cdot c_{\ref{cor:gap_k+}}(k)$.

Let $I$ be an interval of $V(G)$ with $|I| \geqslant N$ and $x_1, \dots, x_{k+1}$ be $k+1$ distinct vertices of $V(G) \setminus I$.
Let $g := \lceil \frac{|I|}{c_{\ref{cor:gap_k+}}(k)} \rceil \geqslant 2$.
If one of the vertices $x_1, \dots, x_k$ has neighborhood gap at least $g$ in $I$, then the corollary holds.

Otherwise, the definition of $g$ gives that
\[
\frac{|I|}{c_{\ref{cor:gap_k+}}(k)} + 1 \geqslant g, \text{ hence, since $g \geqslant 2$, }
|I| \geqslant c_{\ref{cor:gap_k+}}(k)\frac g 2 \geqslant f_{\ref{lem:gap_and_interval}}(g, k).
\]

We can therefore apply \Cref{lem:gap_and_interval} on the vertices $x_1, \dots, x_k$, the interval $I$ and the gap $g$.
Let $I' \subseteq I$ be the subinterval given by \Cref{lem:gap_and_interval}, with $|I'| \geqslant \frac{|I|}{k!} - 5kg$.
Note that
\[
\frac{|I|}{k!} \geqslant 6 \cdot 2^k \cdot g \geqslant 5kg + 3g \text{ because $k \geqslant 1$.}
\]
Hence $|I'| \geqslant 3g$.
Now, partition $I'$ in $I'_l \cup I'_m \cup I'_r$ with $I'_l \prec I'_m \prec I'_r$ and $|I'_l| = |I'_r| = g$. Since all $x_1, \dots, x_k$ have neighborhood gap less than $g$ in $I$, they are adjacent to both $I'_l$ and $I'_r$, and by \Cref{lem:gap_and_interval}, the edges between $x_i$ and $I'_l \cup I'_r$ are in $E_i$.
This implies that $x_{k+1}$ has no neighbor in $I'_m$: indeed, let $i$ be such that there is an edge in $E_i$ incident to $x_{k+1}$ and a vertex of $I'_m$. Then this edge crosses one of the (at least) two edges between $x_i$ and $I'_l \cup I'_r$.
Since $x_{k+1}$ has no neighbor in $I'_m$, and $|I'_m| \geqslant g$, $x_{k+1}$ has neighborhood gap at least $g$ in $I$.
\end{proof}
\begin{lemma}\label{lem:gap_k}
Let $g$ be an integer, $(G,\prec)$ be ordered along a Hamiltonian path $P$ and \notcross{k}, $I$ be an interval of $G$ with $|I| \geqslant f_{\ref{lem:gap_and_interval}}(g, k)$, and let $x_1, \dots, x_k \in V(G) \setminus I$ be $k$ distinct vertices.
Then, either one of the $x_i$ has neighborhood gap at least $g$ in $I$, or $G$ contains an increasing induced path of size $\frac{|I|}{g \cdot k!} - 5k - 1$.
\end{lemma}
\begin{proof}
Let $v_1 \prec \dots \prec v_n$ be the order induced by $P$, $I$ be an interval of $G$, and let $x_1, \dots, x_k \in V(G) \setminus I$ be $k$ distinct vertices.
Consider $E_1 \cup \dots \cup E_k$ a partition of $E(G)$ into non-crossing sets of edges, and assume that all $x_i$ have neighborhood gap less than $g$ in $I$.

Let $I'$ be the interval obtained from \Cref{lem:gap_and_interval}.
We partition $I'$ into $p = \lfloor |I'|/g \rfloor$ intervals $I'_1, \dots, I'_p$ each of size at least $g$, such that $I'_1 \prec I'_2 \prec \dots \prec I'_p$.
Note that, for each $j \leqslant p$, for each $i \leqslant k$, there is an edge of $E_i$ between $x_i$ and $I'_j$.
In particular, for any integers $j,j' \leqslant p$ such that $|j-j'|\geqslant 2$, there is no edge between $I'_j$ and $I'_{j'}$ (in any $E_i$).
Hence, a shortest increasing path from $I'_1$ to $I'_p$ is of size at least $p$ (see \Cref{lem:lifting}).

To conclude, we have $p \geqslant \lfloor |I'|/g \rfloor \geqslant \frac{|I|}{g \cdot k!} - 5k - 1$.
\end{proof}

Our proof will construct trees of bounded cutwidth, which are easier to maintain when seen as colored trees.
For any integer $t$, a \emph{$t$-colored tree} is a tree with edges colored by $[t]$.
Given a rooted tree $T$, an order $\prec$ is \emph{compatible with $T$} whenever for any $u, v \in V(T)$ with $u$ an ancestor of $v$, we have $v \prec u$.
For any rooted tree $T$, we denote by $r(T)$ the root of $T$. 

\begin{figure}[ht]
    \centering
    \includegraphics{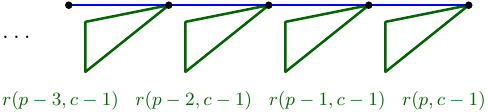}
    \caption{An illustration of the recursive decomposition in the proof of \Cref{lem:cw-diam}.}
    \label{fig:tree-cw}
\end{figure}

To deduce precise bounds on the size of induced paths in the next theorem, we will use the following lemma about the relationship between depth and size in a tree compatible with an order of bounded cutwidth.

\begin{lemma}\label{lem:cw-diam}
Let $(T,\prec)$ be an ordered rooted tree compatible with its order, whose cutwidth as an ordered graph is at most $c$ and whose depth\footnote{The number of vertices on a longest root-to-leaf path.} is at most $p$.
Then, $|V(T)| \leqslant \binom{p+c-1}{c}$.
\end{lemma}

\begin{proof}
Let $r(p,c)$ denote the maximum number of vertices of such a rooted tree.
Because the order is compatible, the root of each subtree is ordered last in the ordering. Consider the first vertex in the ordering and observe that the path between this vertex and the root is increasing and contributes exactly $1$ to each cut (all cuts are hit because the root is the last vertex). We can recursively bound the size of all maximal subtrees which are disjoint from the edges of this path, as shown in \Cref{fig:tree-cw} because their cutwidth is at most $c-1$. We have $r(\cdot,0) = 1$ and $r(1,\cdot)=1$ as base cases, and from induction and the previous observation, we obtain:
\[
r(p,c) \leqslant \sum_{j=1}^{p} r(j,c-1) \leqslant \sum_{j=1}^{p} \binom{j+c-2}{c-1}
= \binom{p+c-1}{c}.\]
The last equality is known as the Hockey-stick identity. It follows from rewriting the term equal to $1$ and applying Pascal's formula inductively.
\end{proof}

\begin{corollary}\label{cor:tree-cw}
Let $T$ be a $k$-colored rooted tree, together with a compatible order $\prec$, such that for any color $j$, the edges of color $j$ form an increasing sequence of increasing paths\footnote{This corresponds to asking for $(T,\prec)$ to have \emph{colored cutwidth $1$} as introduced in \cite{BodlaenderFHWW00}, which proves hardness for computing an ordering that has colored cutwidth $1$ parameterized by the number of colors.}. Then, the depth of $T$ is at least: $$1 - k + \frac{k}{e}(|V(T)|)^{1/k}.$$
\end{corollary}

\begin{proof}
$T$ is the edge union of $k$ ordered subgraphs of cutwidth at most $1$ so it has cutwidth at most $k$. The lower bound on the depth then follows from using \Cref{lem:cw-diam} and the following bound on binomial coefficients: \[\binom{a}{b} < \left(\frac{ae}{b}\right)^b.\]
\end{proof}


We are now ready for the main part of the proof.

\begin{theorem}\label{thm:k-outer-thm}
Let $(G,\prec)$ be an $n$-vertex graph ordered along a Hamiltonian path $P$ and be \notcross{k}.
Then, $(G,\prec)$ contains an induced increasing path of size
\begin{itemize}
    \item $\Omega_k \left( \left( \frac{\log n }{\log \log n}\right)^{1/p} \right)$ when $k$ is even and $p = \frac{k}{2}$
    \item $\Omega_k \left( \left( \log n\right)^{1/p} \right)$ when $k$ is odd and $p= \lceil \frac{k}{2} \rceil$
\end{itemize}
\end{theorem}
\begin{proof}

The proof consists of growing a set of trees, called a \emph{tree surrounding} in the following, around an interval of the vertex ordering $\prec$ in such a way that any path from a root to a leaf is an induced path. For technical reasons, we cannot hope for all trees to have an increase in size at each growing step. However, we will show that the sum of their sizes does increase.
This allows us to obtain at least one tree of large size at the end of the growing process.
This can be achieved by proving the invariants of \Cref{it:1,it:2,it:3,it:special} presented below and using \Cref{cl:aux-k-outer-lemma}.
However, trees of large size do not necessarily have large depth. Hence, we provide a coloring of edges that is a certificate of small cutwidth for our trees. This is achieved by the additional invariants of \Cref{it:4,it:5} presented below.

Note that the trees we grow are not, in general, induced. We only ensure that root-to-leaf paths are induced.

A \emph{$t$-tree-surrounding of an interval $I$} is a tuple $(L_1, \dots, L_p, I, R_1, \dots, R_{t-p})$ where $t \in \{k, k+1\}$, $L_1, \dots, L_p$ as well as $R_1, \dots, R_{t-p}$ are disjoint colored trees, and $I$ is an interval of $V(G)$, that satisfies the following list of assumptions for any $i \in [p]$ and $j \in [t-p]$.

\begin{enumerate}
    \item\label{it:1} $V(L_i) \prec I \prec V(R_j)$.
    \item\label{it:4} The edges of $L_i$ have colors in $[p]$ and the edges of $R_j$ have colors in $\{p+1, \dots, t\}$. 
    \item\label{it:5} In $L_1 \cup \dots \cup L_p \cup R_1 \cup \dots\cup R_{t-p}$, the set of edges of color $i$ induces an increasing sequence of increasing paths, and similarly for the set of edges of color $j + p$.
    Furthermore, $r(L_{i})$ is larger (resp. $r(R_{j})$ is smaller) than any other vertex incident to an edge of color $i$ (resp. $j + p$).
    \item\label{it:2} If $v \in V(L_i) - r(L_i)$ or $v \in V(R_j) - r(R_j)$ then $N[v] \cap I = \varnothing$.
    \item\label{it:3} If $v \in V(L_i) - r(L_i)$ (resp. $v \in R_j - r(R_j)$), the parent of $v$ is the largest (resp. smallest) neighbor of $v$ smaller (resp. larger) than $I$.
    \item\label{it:special} Let $A := \{r(L_a)~:~a \in [p]\} \cup \{r(R_b)~:~b\in [t-p]\}$. There exists $v \in A$ with $N(v) \cap I = \varnothing$.
\end{enumerate}

\begin{claim}\label{cl:aux-k-outer-lemma}
Let $(L_1, \dots, L_{p}, I, R_1, \dots, R_{t-p})$ be a $t$-tree-surrounding of $I$. Then the following holds.

\noindent {\bf If $\mathbf{t = k}$}, then for any positive integer $g$, as long as $|I| \geqslant f_{\ref{lem:gap_and_interval}}(g, k) + 1$, $G$ has an increasing induced path on $\frac{|I|}{g \cdot k!} - 5k - 2$ vertices, or there exists a tree-surrounding $(L'_1, \dots, L'_{p}, I', R'_1, \dots, R'_{t-p})$ of $I' \subseteq I$ such that:
\begin{itemize}
    \item $|I'| \geqslant g$
    \item $\sum_i |V(L_i')| + \sum_j |V(R_j')| = 1 + \sum_i |V(L_i)| + \sum_j |V(R_j)| $.
\end{itemize}
\noindent {\bf If $\mathbf{t = k + 1}$}, then, as long as $|I| \geqslant f_{\ref{cor:gap_k+}}(k) + 1$,  there exists a $t$-tree-surrounding $(L'_1, \dots, L'_{p}, I', R'_1, \dots, R'_{t-p})$ of $I' \subseteq I$ such that
\begin{itemize}
    \item $|I'| \geqslant (|I|-1) / c_{\ref{cor:gap_k+}}(k)$
    \item $\sum_i |V(L_i')| + \sum_j |V(R_j')| = 1 + \sum_i |V(L_i)| + \sum_j |V(R_j)| $.
\end{itemize}

\end{claim}
\begin{proof}
The difference between the two cases lies simply in which of \Cref{cor:gap_k+} or \Cref{lem:gap_k} we are allowed to apply. In particular, most of the proof is independent of whether $t = k$ or $t = k+1$.

By \Cref{it:special}, there is a $u \in \{r(L_i)~:~i \in [p]\} \cup \{r(R_j)~:~j\in [t-p]\}$ such that $N(u) \cap I = \varnothing$.
We assume that $u = r(L_a)$ for some $a$, but a symmetric proof would work for $u = r(R_b)$.
Consider the largest neighbor $v$ of $u$ that is smaller than $I$ (it must exist because of the neighbor of $u$ on the Hamiltonian path).
Note that by definition, $v \prec I$, so $v \not \in V(R_1) \cup \dots \cup V(R_{t-p})$.

We need to distinguish between the case where $v$ is in $V(L_1) \cup \dots \cup V(L_p)$ and the case where it is not, since in the former, we need to change our trees in a precise way. Hence, we exhibit $L'_1, \dots, L'_p$ and $R'_1, \dots, R'_{t-p}$ in both cases, before proving the existence of either a long induced path or an interval $I'$ satisfying the claim.

If $v$ is not in $V(L_1) \cup \dots \cup V(L_p)$, then we consider $L'_1, \dots, L'_p, R'_1, \dots, R'_{t-p}$ where $R'_i := R_i$ for all $i$, $L'_i := L_i$ if $i \neq a$ and $L'_a$ is obtained from $L_a$ by defining $v$ as ancestor of $r(L_a)$, with the edge $r(L_a) v$ of color~$a$ (color $b + p$ for the symmetric setting).
Note that all trees are disjoint.
For any interval $I' \subseteq I$, \Cref{it:1,it:2,it:3,it:4,it:5} are clearly satisfied, since all trees but $L_a$ are unchanged, and $u = r(L_a)$ is not adjacent to $I$.
Hence, it remains to exhibit $I' \subseteq I$ satisfying \Cref{it:special}. This is handled uniformly after the next case.

Otherwise, there exists $b$ such that $v \in V(L_b)$.
We consider $L'_1, \dots, L'_p, R'_1, \dots, R'_{t-p}$ with $R_i' := R_i$ for all $i$, $L_i' := L_i$ when $i \not \in \{a, b\}$, $L_a'$ to be the single-vertex tree containing the smallest vertex larger than $\cup_i V(L_i)$, and $L'_b$ is obtained from $L_a \cup L_b$ by setting $v$ as ancestor of $r(L_a)$ with the edge $r(L_a) v$ of color~$a$.
We start by proving that for any interval $I' \subseteq I - m$, with $m$ the smallest vertex of $I$, \Cref{it:1,it:2,it:3,it:4,it:5} hold.
Since $r(L_a') \preceq m$, we have $L_a' \prec I'$ (thus \Cref{it:1} holds). \Cref{it:2,it:3,it:4} clearly hold, and for \Cref{it:5}, it is sufficient to observe that it holds for color $a$ and that the sets of edges and roots assigned to other colors did not change.
By assumption, we know that $r(L_a)$ is larger than any vertex incident to an edge of color $a$, except for the edge $r(L_a) v$. Since $r(L_a) \prec v$, we have that $r(L_a) v$ is the only edge of color $a$ incident to $v$. Hence, edges of color $a$ still make an increasing sequence of increasing paths. Furthermore, $r(L_a')$ is larger than $v$, so larger than any vertex incident to an edge of color $a$.
Again, it remains to exhibit $I' \subseteq I - m$ satisfying \Cref{it:special}.

\noindent {\bf If $\mathbf{t = k}$}, let $g$ be an integer. Since $|I| \geqslant f_{\ref{lem:gap_and_interval}}(g, k) + 1$, we have $|I - m| \geqslant f_{\ref{lem:gap_and_interval}}(g, k)$,
and we can apply \Cref{lem:gap_k} on $A := \{r(L_1'), \dots, r(L_p'),$ $r(R_1'), \dots, r(R_{t-p}')\}$ as the set of $k$ vertices, the interval $I - m$, and gap value $g$.
Either $G$ contains an increasing induced path of size $\frac{|I| - 1}{g \cdot k!} - 5k -1\geqslant \frac{|I|}{g \cdot k!} - 5k - 2$ and the claim holds, or there exist $a \in A$ and $I' \subseteq I - m $ such that $N[a] \cap I' = \varnothing$, and $|I'| \geqslant g$, proving \Cref{it:special}.

\noindent {\bf If $\mathbf{t = k + 1}$}, we have $|I - m| = |I| - 1 \geqslant f_{\ref{cor:gap_k+}}(k)$. Hence we apply \Cref{cor:gap_k+} on the $k+1$ roots of our trees $A := \{r(L_1'), \dots, r(L_p'),$ $r(R_1'), \dots, r(R_{t-p}')\}$ and the interval $I - m$.
There exist $a \in A$ and $I' \subseteq I - m $ such that $N[a] \cap I' = \varnothing$, and $|I'| \geqslant (|I|-1)/c_{\ref{cor:gap_k+}}(k)$, proving \Cref{it:special}.

\end{proof}

With this claim, we are able to prove the theorem.
We will first prove in detail the case where $k$ is even. Then the case $k$ odd will be very similar.
In both cases, we assume that $n$ is large enough compared to some function of $k$. The following suffices: $(\log\log n)^{1/p} \geqslant 24\cdot k!\cdot 2^k$.

\noindent{\bf Case $\mathbf{k}$ even.}
Let $p := \frac k 2$.
If $G$ contains an increasing induced path of size $\left(\frac{\log n}{\log \log n}\right)^{1/p}$, then the theorem holds. Now, assume $G$ does not contain such a path.

We build a sequence of tree-surroundings of intervals $(\mathcal{T}_t)_{t\geqslant 1}$ satisfying that:
\begin{itemize}
    \item $\mathcal{T}_t$ is of the form $(L_1, \dots, L_p, I_t, R_1, \dots, R_{k-p})$; and
    \item $|I_t| \geqslant \left\lfloor\frac{n}{2 (\log n)^{t}} \right\rfloor$; and
    \item The sum of the sizes of the trees of $\mathcal{T}_t$ is at least $t$.
\end{itemize}

For the base case, let $L_i$ be the one-vertex tree containing the $i$-th vertex of $P$ for each $i \in [p]$ and $R_j$ be the one-vertex tree containing the $(n-j + 1)$-th vertex of $P$ for $j \in [k - p]$.

It remains to find an interval $I_1$ for which $(L_1, \dots, L_p, I_1, R_1, \dots, R_{k-p})$ is a tree-surrounding.
Let $I = [p + 1, n - (k - p)]$.
For any interval $I' \subseteq I$, we have $V(L_i) \prec I' \prec V(R_j)$ by construction, and obtain immediately \Cref{it:1,it:2,it:3,it:4,it:5}.
We must now satisfy \Cref{it:special}.
Let $g_0 := \frac{n}{2 \log n}$. Since $\log n \geqslant 6 k! 2^k$, we have $|I| \geqslant k! 2^k 6 g_0 \geqslant f_{\ref{lem:gap_and_interval}}(g_0, k)$.
We apply \Cref{lem:gap_k} on the set of vertices $A := \{r(L_1), \dots, r(L_p), r(R_1), \dots, r(R_{k-p})\}$, the interval $I$, and the gap $g_0$.
The graph $G$ does not have an increasing induced path of size
\[
\frac{n - k}{g_0 \cdot k!} - 5k - 2 \geqslant \frac{n}{2 g_0 \cdot k!} - 5k - 2 \geqslant \frac{\log n}{k!} - 5k - 2 \geqslant (\log n/\log \log n)^{1/p},
\]
since $(\log \log n)^{1/p} \geqslant k!(5k - 2)$.
Hence, there exists $a \in A$ of neighborhood gap at least $g_0$ in $I$. We let $I_1 \subseteq I$ be any interval of $g_0$ vertices not adjacent to $a$, and observe that $\mathcal{T}_1 := (L_1, \dots, L_p, I_1, R_1, \dots, R_{k-p})$ is a tree-surrounding.

Assume now that $\mathcal{T}_t$ is defined. We want to apply \Cref{cl:aux-k-outer-lemma} with target interval size $g_t := \left\lfloor \frac{n}{2 (\log n)^{t+1}} \right\rfloor$.
First, note that as long as $g_t \geqslant 1$,
\[
\frac{|I_t|}{g_t} \geqslant \lfloor \frac{n}{2 (\log n)^t}  \rfloor \cdot \frac{2 (\log n)^{t+1}}{n} \geqslant \log n - 1 \geqslant 6 k! 2^k + 1.
\]
In addition, since $G$ does not contain an increasing induced path of size $\frac{|I_t|}{g_t \cdot k!} - 5k - 2 \geqslant \frac{\log n}{k!} - 5k - 3 \geqslant (\log n/\log \log n)^{1/p}$ (using again that $(\log \log n)^{1/p} \geqslant k!(5k - 2)$), \Cref{cl:aux-k-outer-lemma} implies there is a $k$-tree-surrounding $\mathcal{T}_{t+1}$ such that $|I_{t+1}| \geqslant g_t$ and the sum of the sizes of the trees of $\mathcal{T}_{t+1}$ is at least $t+1$.
This sequence is well defined as long as $g_t \geqslant 1$.
In particular, by taking $t = \frac{\log n}{2 \log \log n}$, we have 

\[
\begin{aligned}
\log (g_t) & =  
\log n - \log \left(2(\log n)^{t+1} \right)\\ & =  \log n - 1 - (t+1) \log \log n\\ &
=  \frac{\log n}{2} -1 - \log \log n  >  0 \\
\end{aligned}
\]

Hence, for $t = \frac{\log n}{2 \log \log n}$, $\mathcal{T}_t$ contains a tree with at least $\frac{t}{k}$ vertices, say $T$.
By \Cref{it:3}, each root-to-leaf path of $T$ is induced; in addition,
\Cref{it:4} ensures that $T$ is $p$-colored, \Cref{it:3} ensures that either $\prec$ or $\succ$ is compatible with $T$, and \Cref{it:5} ensures that monochromatic edges of $T$ form an increasing sequence of increasing paths.
Hence, by \Cref{cor:tree-cw}, $T$ has depth at least $1 - p + \frac{p}{e}(|V(T)|)^{1/p}$.
Hence, $G$ contains an increasing induced path of size at least
\[
1 - p + \frac{p}{e}\left(\frac{\log n}{2k \log \log n}\right)^{1/p} = \Omega_p \left( \left( \frac{\log n }{\log \log n}\right)^{1/p} \right).
\]

\noindent{\bf Case $\mathbf{k}$ odd.}
Let $p := \lceil \frac k 2 \rceil$.
We build a sequence of $(k+1)$-tree-surroundings of intervals $(\mathcal T_t)_{t \geqslant 1}$ satisfying that:
\begin{itemize}
    \item $\mathcal T_t$ is of the form $(L_1, \dots, L_p, I_t, R_1, \dots, R_p)$; and
    \item $|I_t| \geqslant \frac{|I|}{c^t}-1$ with $c := c_{\ref{cor:gap_k+}}(k)$; and
    \item the sum of the sizes of the trees of $\mathcal T_t$ is at least $t$.
\end{itemize}
The trees of the base case are built in the exact same way. To find $I_1$, note that $|I| = n - (k + 1) \geqslant f_{\ref{cor:gap_k+}}(k)$, and so one can apply \Cref{cor:gap_k+} on the set of $k+1$ vertices $A := \{r(L_1), \dots, r(L_p), r(R_1), \dots, r(R_p)\}$ and the interval $I$. Hence there is $a \in A$ of neighborhood gap at least $|I|/c$ in $I$, and we can choose $I_1 \subseteq I$ an interval of $|I|/c$ vertices not adjacent to $a$.

Assume now that $\mathcal T_t$ is defined, then by \Cref{cl:aux-k-outer-lemma}, as long as $|I_t| \geqslant f_{\ref{cor:gap_k+}}(k) + 1$, we find a $(k+1)$-tree-surrounding $\mathcal T_{t+1}$ such that $|I_{t+1}| \geqslant \frac{|I_t|-1}{c} \geqslant \frac{|I|/c^t-2}{c} \geqslant \frac{|I|}{c^{t+1}}-1$ (since $c \geqslant 2$).

By taking $t = \left\lfloor \frac{\log |I|}{2 \log c} \right\rfloor$, we have
\[
\begin{aligned}
\log (|I_t|+1) & \geqslant  \log |I| - t \cdot \log c\\
& \geqslant  \log |I| - \frac{\log |I|}{2\log c} \log c\\
& =  \frac{\log |I|}{2} >  \log(f_{\ref{cor:gap_k+}}(k) + 1).\\
\end{aligned}
\]

Hence, $\mathcal T_t$ is well defined for $t = \lfloor\frac{\log |I|}{2 \log c}\rfloor$, and contains a tree with at least $\lceil \frac{t}{k+1} \rceil$ vertices, say $T$. By the same computations as above, one obtains that $T$ has depth at least $1 - p + \frac p e \left( |V(T)| \right)^{1/p}$, and so $G$ contains an increasing induced path of size at least $\Omega_p\left( (\log n)^{1/p} \right)$.
\end{proof}

\section{Tight examples}\label{sec:upper-bounds}

We provide in this section a construction of $n$-vertex $k$-non-crossing ordered graphs whose longest induced paths have size $\Theta_k\left(\left( \frac{\log n}{\log \log n}\right)^{2/k}\right)$ (resp. $\Theta_k\left((\log n)^{\frac{2}{k+1}}\right)$), when $k$ is even (resp. odd), matching the bounds of \Cref{thm:k-outer-thm}.

Let us introduce some terminology.
If $(G, \prec)$ is an ordered graph, the first vertex and last vertex of $(G, \prec)$ are its \emph{extremities}.
Given two ordered graphs $G$ and $H$, we define the \emph{gluing of $G$ and $H$}, denoted $G \cdot H$, as the ordered graph on $|V(G)| + |V(H)| - 1$ vertices, where the first $|V(G)|$ vertices induce a copy of $G$, and the last $|V(H)|$ vertices induce a copy of $H$. In particular, the last vertex of the copy of $G$ is the first vertex of the copy of $H$.

For each integer $p \geqslant 0$, we build the ordered (outerplanar) graph $U_p$ with $n =  2^{p} + 1$ vertices as follows:
$U_0$ is the ordered path on two vertices, and for each integer $p$, $U_{p+1}$ is obtained by adding an edge between the two extremities of $U_p \cdot U_p$.
One can easily check by induction that for any $p$, $U_p$ is \notcross{1}, and that it contains an increasing Hamiltonian path.

\begin{lemma}\label{lem:Up_IP}
For any integer $p$, the ordered graph $U_p$ has no induced path on $2 p + 3$ vertices.
\end{lemma}
\begin{proof}
Let $f : p  \mapsto p + 2$.
We prove the following stronger property by induction: any induced path of $U_p$ ending at an extremity of $U_p$ is of size at most $f(p)$, and any induced path of $U_p$ is of size at most $\max(f(p), 2 f(p - 1))$.

The graph $U_0$ has no induced path on $3$ vertices.
For any integer $p$, if $U_p$ satisfies the induction hypothesis, then an induced path of $U_{p+1}$ is contained in $[1, 2^p +1]$, in $[2^p + 1, 2^{p+1} +1]$, or it contains a vertex of $\{1, 2^p + 1, 2^{p+1} + 1\}$.
In the first two cases, we can immediately apply induction.
Hence, we assume we are in the last case, and let $P = a_1 \dots a_k$ be an induced path with $a_j$ the first vertex of $P$ in $\{1, 2^p + 1, 2^{p+1} + 1\}$.
Clearly, $\{a_1, \dots, a_j\}$ is contained in $[1, 2^p +1]$ or in $[2^p +1, 2^{p+1} + 1]$. So the path $a_1 \dots a_j$ is of size $\leqslant f(p)$ by induction.
With a symmetric argument, we define $a_\ell$ to be the last vertex of $P$ in $\{1, 2^p + 1, 2^{p+1} + 1\}$, and obtain that $a_\ell \dots a_k$ is of size $\leqslant f(p)$.
Since $a_j$ and $a_\ell$ are equal or adjacent, we have $\ell - j \leqslant 1$, so $k \leqslant 2 f(p)$.
Furthermore, if $j = 1$ or $\ell = k$, then $k \leqslant f(p) + 1$, which implies $k \leqslant f(p+1)$.
\end{proof}

We also build the graph $J_p$, which is the construction of \cite{EsperetLM17} ordered along a Hamiltonian path, with $n = p! + 1$ vertices as follows:
$J_1$ is the ordered path on two vertices, and for each positive integer $p$, we first build $H_{p + 1}$ by gluing $p+1$ copies of $J_p$ denoted $J_p^i$ for $i \in [p+1]$.
We let $c_i$ be the first vertex of $J_p^i$ for $i \in [p+1]$, and $c_{p+2}$ be the last vertex of $J_p^{p+1}$.
$J_{p + 1}$ is then the graph obtained from $H_{p+1}$ by adding the edges $c_1 c_i$, $c_i c_{p+2}$ and $c_1c_{p+2}$ for $i \in \{2, \ldots, p+1\}$.
In particular, the two extremities of a copy of $J_p$ are adjacent.
One can easily check that $J_p$ is \notcross{2} and contains an increasing Hamiltonian path.
We use the following bound that can be found in \cite{EsperetLM17} and add a short proof for completeness.
\begin{lemma}[\cite{EsperetLM17}]\label{lem:Jp_IP}
For any positive integer $p$, the ordered graph $J_p$ has no induced path on $3p+1$ vertices.
\end{lemma}
\begin{proof}
Let $f : p  \mapsto p + 1$.
We prove the following stronger property by induction: any induced path of $J_p$ ending at an extremity of $J_p$ is of size at most $f(p)$, and any induced path of $J_p$ is of size at most $\max(f(p), 2 f(p - 1) + p)$.

The graph $J_1$ has no induced path on $3$ vertices.
Assume that $J_p$ satisfies the induction hypothesis, and let $P = v_1, \dots, v_\ell$ be an induced path of $J_{p+1}$. Let $c_1,\dots,c_{p+2}$ denote the extremities of copies of $J_p$ in $J_{p+1}$.
Let $i$ be the first index such that $v_i$ is of the form $c_t$, that is $v_i$ is an extremity of a copy of $J_p$ inside $J_{p+1}$. Similarly, let $j$ be the last index such that $v_j$ is of the form $c_t$.
Note that, if $v_i$ and $v_j$ do not exist, then $P$ is an induced path in a copy of $J_p$, and so the statement holds.

By induction, we have that $\ell - j + 1$ and $i$ are both at most $f(p)$, and it remains to bound the size of the path $v_{i+1} \dots v_{j - 1}$.
Assume for the sake of contradiction that there exists a vertex $v_{k}$ with $i < k < j$ and $v_k$ is not of the form $c_t$.
In particular, there is a copy $J_p^a$ of $J_p$ containing $v_k$.
Let $k^-, k^+$ be, respectively, the largest index smaller than $k$ such that $v_{k^-}$ is of the form $c_t$, and the smallest index larger than $k$ such that $v_{k^+}$ is of the form $c_t$.
Since there are no edges from the interior of $J_p^a$ to the rest of $J_{p+1}$, both $v_{k^-}$ and $v_{k^+}$ are extremities of $J_p^a$, and so are adjacent by construction of $J_{p+1}$; a~contradiction.

Hence, $v_i, \dots, v_j$ contains only vertices of the form $c_t$, implying that it is of order at most $p+2$.
Hence, $P$ contains at most $2 f(p) + p$ vertices.
On the other hand, if the last vertex of $P$ is $c_1$ or $c_{p+2}$, it means that $j = \ell$ and $v_j \in \{c_1, c_{p+2}\}$. In particular, $v_j$ is adjacent to $v_i$,  thus $j = i + 1$ and $P$ contains at most $f(p) + 1 = f(p + 1)$ vertices.
\end{proof}

Now, for each pair of non-negative integers $p$ and $k$, we will construct the ordered graphs $G(k, p)$. 

The construction is made by induction as follows.
\begin{itemize}
\item For any integer $k \geqslant 0$, $G(k, 0)$ is the ordered path on 2 vertices.
\item For any integer $p \geqslant 0$, $G(0, p)$ is the graph $U_{p}$.
\item For any integer $p \geqslant 1$, $G(1, p)$ is the graph $J_{p}$.
\item Given two integers $k$ and $p$, we define the graph $H(k + 2, p)$ as the gluing $C_1 \cdot \ldots \cdot C_d$, where $d := |V(G(k, p)) | - 1$, and each $C_i$ is a copy of $G(k + 2, p-1)$.
We denote by $c_i$ the first vertex of $C_i$ for $i \in [d]$, and by $c_{d+1}$ the last vertex of $C_d$. 
\item $G(k + 2, p)$ is obtained from $H(k + 2, p)$ by adding:
\begin{itemize}
\item a complete bipartite graph between $\{c_1, c_{d+1}\}$ and $\{c_i~:~ 1 < i < d+1\}$; and
\item a copy of $G(k, p)$ on $\{c_1, \dots, c_{d+1}\}$.
\end{itemize}
\end{itemize}

\newcommand{\Int}[1]{#1^o}
In the rest of the section, we will denote by $X(k + 2, p)$ the set of vertices $\{c_1, \dots, c_{d+1}\}$. Given an interval $I$, the \emph{interior of $I$}, denoted $\Int{I}$, is the interval consisting of elements of $I$ without its extremities. We extend this notation to subgraphs induced by intervals such as $C_i$.

\begin{lemma}
For any two integers $k$ and $p$, for any two copies $C_i$ and $C_j$ with $i < j$ of $G(k, p-1)$ in the construction of $H(k, p)$, there is no edge between $\Int{C_i}$ and $\Int{C_j}$ in $G(k, p)$.
\end{lemma}
\begin{proof}
Indeed, a gluing does not add edges, and all the edges in $E(G(k, p)) - E(H(k, p))$ are between vertices of $X(k, p)$.
\end{proof}

\begin{lemma}
$G(k, p)$ is \notcross{(k + 1)}.
\end{lemma}
\begin{proof}
We prove the property by following the induction of the construction of $G(k, p)$.
Note that $G(k, 0)$ and $G(0, p)$ are \notcross{1} by definition, and $G(1, p)$ is \notcross{2} by definition.
Assume that for some integers $k, p$, we have that for any pair $(k', p')$ lexicographically smaller than $(k + 2, p)$, the ordered graph $G(k', p')$ is \notcross{(k'+1)}.

The ordered graph $H(k + 2, p)$ is a gluing of \notcross{(k+3)} graphs. Hence, it is \notcross{(k + 3)}.
For any two edges $e \in E(H(k + 2, p))$ and $f \in E(G(k + 2, p)) - E(H(k + 2, p))$, $e$ and $f$ are non-crossing: indeed, $f$ is an edge between extremities of copies of $G(k + 2, p-1)$, and $e$ is inside a copy of $G(k + 2, p-1)$.
Hence, it suffices to prove that there exists a partition of $E(G(k + 2, p)) - E(H(k + 2, p))$ into $k + 3$ non-crossing sets of edges.
This is true by induction: consider a partition of the edges of the copy of $G(k, p)$ spanning $X(k + 2, p)$ into $k + 1$ sets of non-crossing edges, and finish the partition of $E(G(k+2, p)) - E(H(k+2, p))$ by adding the two sets $\{c_1c_i~:~ 1 < i \leqslant d\}$ and $\{c_ic_{d+1}~:~ 1 < i \leqslant d \}$.
\end{proof}

\begin{lemma}
We have $|V(G(k, p))| = 2^{\binom{p+k/2}{k/2+1}} + 1$ when $k$ is even; and $|V(G(k, p))| \geqslant  \left\lfloor \frac{p}{2} \right\rfloor^{\binom{ \lfloor p/2 \rfloor + \lfloor k/2 \rfloor +1}{\lfloor k/2 \rfloor + 1}}$ when $k$ is odd.
\end{lemma}
\begin{proof}
Let $n_{k, p} := |V(G(k, p))|$. The proof goes by induction on pairs $(k, p)$, where in the odd case, we will prove the following equality implying the desired statement: \[n_{k, p}-1 = \prod_{i=0}^{p-1}(p-i)^{\binom{i+\lfloor k/2 \rfloor}{i}} \geqslant  \left\lfloor \frac{p}{2} \right\rfloor^{\sum_{i=0}^{\lfloor p/2 \rfloor}\binom{ i + \lfloor k/2 \rfloor}{\lfloor k/2 \rfloor}} = \left\lfloor \frac{p}{2} \right\rfloor^{\binom{ \lfloor p/2 \rfloor + \lfloor k/2 \rfloor +1}{\lfloor k/2 \rfloor + 1}} .\]
The inequality follows from considering only the first $\lfloor p/2 \rfloor +1$ terms, whose bases are at least $\lfloor p/2 \rfloor$. The equality on the right is obtained by the Hockey-stick identity.

By definition, we have $n_{k, 0} = 2 = 2^{\binom{k/2}{k/2+1}} + 1$ and $n_{k,0}=1^{\binom{\lfloor k/2 \rfloor}{0}}+1$, $n_{0, p} = 2^{p} + 1 = 2^{\binom{p}{1}} + 1$, and $n_{1, p} -1= p! = \prod_{i=0}^{p-1} (p-i)^{\binom{i}{i}}$.

Now assume $k \geqslant 2$.
By definition, $V(G(k, p)) = V(H(k, p))$, and $H(k, p)$ is obtained by the gluing of $n_{k-2, p} - 1$ copies of $G(k, p-1)$, hence we have:
\[
n_{k, p} = (n_{k-2, p} - 1) (n_{k, p-1}) - (n_{k - 2, p} - 2) = (n_{k-2, p} - 1) (n_{k, p-1} - 1) + 1. 
\]

We now apply the induction hypothesis. When $k$ is even, we have:
\[n_{k, p} = 2^{\binom{p + k/2 -1}{k/2}} 2^{\binom{p + k/2 -1}{k/2+1}} + 1 = 2^{\binom{p + k/2}{k/2+1}} + 1.\]
When $k$ is odd, we have: 
    \[
    \begin{aligned}
    n_{k, p} -1 & = \left(\prod_{i=0}^{p-1}(p-i)^{\binom{i+\lfloor k/2 \rfloor-1}{i}}\right)\cdot\left(\prod_{i=0}^{p-2}(p-1-i)^{\binom{i+\lfloor k/2 \rfloor}{i}}\right) \\
    & = \left(\prod_{i=0}^{p-1}(p-i)^{\binom{i+\lfloor k/2 \rfloor-1}{i}}\right)\cdot\left(\prod_{i=0}^{p-1}(p-i)^{\binom{i+\lfloor k/2 \rfloor-1}{i-1}}\right) \\
    & = \prod_{i=0}^{p-1}(p-i)^{\binom{i+\lfloor k/2 \rfloor-1}{i}+\binom{i+\lfloor k/2 \rfloor-1}{i-1}}=\prod_{i=0}^{p-1}(p-i)^{\binom{i+\lfloor k/2 \rfloor}{i}}\\
    \end{aligned}
    \]
\end{proof}

\begin{lemma}\label{lem:contained}
If $P := v_1 \dots v_l$ is an induced path of $G(k, p)$, such that $v_1 \in X(k, p)$ and $v_2 \not\in X(k, p)$, then $\{v_2, \dots, v_l\} \subseteq \Int{C_i}$ for some $i$. 
\end{lemma}
\begin{proof}
By definition, the edge $v_1v_2$ is in $E(H(k, p))$. Hence, there is a copy $C_i$ of $G(k, p-1)$ such that $v_1$ and $v_2$ are two vertices of $C_i$. 
Note that the neighborhood $N(\Int{C_i})$ is exactly the set of extremities of $C_i$. Hence, it is covered by $v_1$ which implies that $\{v_2, \dots, v_l\}$ is contained in $\Int{C_i}$.
\end{proof}

\begin{lemma}\label{lem:large_path}
If $P := v_1 \dots v_l$ is an induced path of $G(k, p)$ and $v_1$ and $v_l$ are in $X(k, p)$, then $P$ is contained in $X(k, p)$.
\end{lemma}
\begin{proof}
Suppose, for the sake of contradiction, that $P$ is not contained in $X(k, p)$. In particular, there exists a first index $i$ with $v_i \in X(k, p)$ and $v_{i+1} \not \in X(k, p)$.

By \Cref{lem:contained}, this implies that $\{v_{i+1}, \dots, v_l\} \cap X(k, p) = \varnothing$; a~contradiction since $v_l \in X(k, p)$.
\end{proof}

\begin{lemma}
$G(k, p)$ does not have an induced path of size $(\lfloor k/2 \rfloor+1)(3p + 4)$. 
\end{lemma}
\begin{proof}
Let $f(k, p) := (\lfloor k/2 \rfloor+1)(3p + 4)$.
We prove by induction on $(k, p)$ that any induced path of $G(k, p)$ starting or ending at an extremity of $G(k, p)$ is of size $\leqslant p + 2$, and any induced path of $G(k, p)$ is of size $< f(k, p)$.

If $k = 0$ or $k = 1$, the statement holds directly as proved in \Cref{lem:Up_IP} and \Cref{lem:Jp_IP}. If $p=0$, the statement is trivial. Also, paths strictly contained in a copy of $G(k,p-1)$ are bounded by the inductive statement applied to the copy. 

Let $P = v_1 \dots v_\ell$ be an induced path of $G(k, p)$ with a vertex in $X(k,p)$.
Consider $i$ the first index such that $v_i \in X(k, p)$.
By \Cref{lem:contained} applied to the path $Q := v_i \dots v_1$, there exists a copy of $G(k, p-1)$ containing $Q$. Hence, by induction, $|V(Q)| \leqslant p + 1$.
Symmetrically, let $j$ be the last index with $v_j \in X(k, p)$. Again the path $v_j \dots v_\ell$ is of order at most $p+1$.
It remains to obtain the size of $M := v_i \dots v_j$, which by \Cref{lem:large_path} is contained in $X(k, p)$. If $M$ contains $c_1$ or $c_{d+1}$, then $|M| \leqslant 3$ and otherwise $M$ is induced in a copy of $G(k-2, p)$.
Thus, we have $|M| \leqslant f(k-2, p)$.
From this, one can derive that: 
\[
|V(P)| \leqslant f(k-2, p) + 2 (p + 1) - 2 \leqslant \lfloor k/2 \rfloor(3p + 4) + 2p <
f(k, p).
\]
In the particular case where $v_1$ is the first vertex of $G(k, p)$, none of the vertices $v_3, \dots, v_\ell$ are in $X(k, p)$ since this vertex set is dominated by $v_1$. By \Cref{lem:contained}, $v_2, \dots, v_\ell$ are all contained in a copy of $G(k, p-1)$; hence, the size of $P$ is at most $p + 1 + 1 = p + 2$.
\end{proof}

\begin{corollary}
For any integer $k$, there exist $n$-vertex graphs ordered along a Hamiltonian path that are
\begin{itemize}
    \item \notcross{(2k)} and do not contain any induced path of size
    \[ \Omega_k \left( \left( \frac{\log n}{\log \log n} \right)^{1/k} \right).\]
    \item \notcross{(2k + 1)} and do not contain any induced path of size
    \[ \Omega_k \left((\log n)^{1/(k+1)}\right).\]
\end{itemize}
\end{corollary}

\bibliography{refs}
\bibliographystyle{alpha}

\end{document}